\documentclass[12pt, reqno]{amsart}
\usepackage{amsmath,amsfonts,amssymb,amsthm,enumerate,appendix,caption}
\usepackage{cmap,mathtools}
\usepackage{paralist,bm}
\usepackage{graphics} 
\usepackage{epsfig} 
\usepackage{graphicx}
\usepackage{braket}
\usepackage{epstopdf}
 \usepackage[colorlinks=true]{hyperref}
\usepackage{xcolor}
\definecolor{myblue}{rgb}{0.0, 0.0, 1.0}
\definecolor{mygreen}{rgb}{0.01,0.75,0.20}
\hypersetup{ linkcolor=myblue, colorlinks=true, urlcolor=black, citecolor=red}
\usepackage{hyperref}

\newtheorem{theorem}{Theorem}[section]
\newtheorem{corollary}[theorem]{Corollary}

\newtheorem{lemma}[theorem]{Lemma}
\newtheorem{proposition}[theorem]{Proposition}

\newtheorem{definition}[theorem]{Definition}

\newtheorem{remark}[theorem]{Remark}

\theoremstyle{definition}

\usepackage[margin=1in]{geometry}

\numberwithin{equation}{section}

\newcommand{\dx}{\,\mathrm{d}x}

\newcommand{\dt}{\,\mathrm{d}t}

\def\ga{\alpha}     \def\gb{\beta}       \def\gg{\gamma}
       \def\gd{\delta}      
                         \def\vge{\varepsilon}
       \def\vgf{\varphi}    
            \def\gl{\lambda}
\def\gm{\mu}        \def\gn{\nu}         
            
       \def\gt{\tau}

     \def\Gd{\Delta}      

\def\Gl{\Lambda}          
\def\Gw{\Omega}              

\DeclarePairedDelimiter\norm{\lVert}{\rVert}%
\makeatletter
\let\oldnorm\norm
\def\norm{\@ifstar{\oldnorm}{\oldnorm*}}

\newcommand{\al} {\alpha}

\newcommand{\De} {\Delta}

\newcommand{\Om} {\Omega}

\newcommand{\la} {\lambda}
\newcommand{\La} {\Lambda}

\newcommand{\ra} {\rightarrow}

\newcommand\restr[2]{{
  \left.\kern-\nulldelimiterspace 
  #1 
  \right|_{#2} 
  }}

\def\w{{\widetilde w}}

\def\w2{{W^{1,2}_0(\Om)}}
\def\hh2{{H^1_0(\Om)}}

\def\C{{\mathcal C}}

\def\N{{\mathbb N}}
\def\F{{\mathcal F}}

\def\R{{\mathbb R}}

\def\({{\Big(}}
\def\){{\Big)}}

\def\ws2{{\F_{\frac{N}{2}}}}

\def\c1{{\C_c^1}}

\def\dt{{\rm d}t}

\def\dx{{\rm d}x}

\newcommand{\Hmm}[1]{\leavevmode{\marginpar{\tiny%
			$\hbox to 0mm{\hspace*{-0.5mm}$\leftarrow$\hss}%
			\vcenter{\vrule depth 0.1mm height 0.1mm width \the\marginparwidth}%
			\hbox to
			0mm{\hss$\rightarrow$\hspace*{-0.5mm}}$\\\relax\raggedright #1}}}

\begin{document}
\title[Lower bound for the weighted-Hardy constant]{A lower bound for the weighted-Hardy constant for domains satisfying a uniform exterior cone condition}
	
	\author {Ujjal Das}
	
	\address {Ujjal Das, Department of Mathematics, Technion - Israel Institute of
		Technology,   Haifa, Israel}
	\email {ujjaldas@campus.technion.ac.il}
	\author{Yehuda Pinchover}
	\address{Yehuda Pinchover,
		Department of Mathematics, Technion - Israel Institute of
	Technology,   Haifa, Israel}
	\email{pincho@technion.ac.il}
	\begin{abstract}
		We consider weighted Hardy inequalities involving the distance function to the boundary of a domain in the $N$-dimensional Euclidean space with nonempty boundary. We give a lower bound for the corresponding best Hardy constant for a domain satisfying a uniform exterior cone condition. This lower bound depends on the aperture of the corresponding infinite circular cone.

		\medskip
		
		\noindent  2000  \! {\em Mathematics  Subject  Classification.}
		Primary  \! 49J40; Secondary  35B09, 35J62.\\[1mm]
	\end{abstract}
\maketitle
\section{Introduction}
Let $N \geq 2$ and $\Gw \subsetneq \R^N$ be a domain. Denote by $\delta_{\Gw}(x)$ the distance of  a point $x\in \Gw$ to $\partial \Gw$. 
Fix $p \in (1,\infty)$ and $\alpha\in \mathbb{R}$. 
We say that the $L^{\al,p}$-{\it{Hardy inequality}} (or the {\it weighted Hardy inequality}) is satisfied in $\Om$ if there exists $C>0$ such that
\begin{align} \label{Lp_Hardy}
  \int_{\Om} |\nabla \varphi|^p \ \delta_{\Om}^{-\al} \dx \geq  C  \int_{\Om} |\varphi|^p \ \delta_{\Om}^{-(\al+p)} \dx \qquad  \forall \varphi \in C_c^{\infty}(\Om).  
\end{align}
The one-dimensional weighted Hardy inequality was proved by Hardy (see, \cite[p.~329]{Hardy}). 
For $\al=0$, the above inequality is often called the {\it geometric Hardy inequality} for domains with boundary. A review for this case is presented in \cite{Balinsky}. The validity of \eqref{Lp_Hardy} indeed depends on $\alpha$ and the domain $\Gw$.
For instance, if $\al+p \leq 1$, then \eqref{Lp_Hardy} does not hold on bounded Lipschitz domains \cite{Leherback1}. However, it holds on $C^{1,\gamma}$-{\it{exterior domains}} with $\gamma \in (0,1]$, when $\al+p<1$ \cite[Corollary 7.3]{DDP}.
On the other hand, for $\al+p>1$, \eqref{Lp_Hardy} is established  for various types of domains: bounded Lipschitz domains \cite{Necas}, domains with H\"older boundary \cite{Kufner},  unbounded John domains \cite{Leherback2}, domains having uniformly $p$-fat complement \cite{Lewis, Wannebo2}, see also the references therein. Set
\begin{align} \label{Lp_Hardy_const}
 \mathbb{H}_{\al,p}(\Om) =\mathbb{H}_{\al,p}:=\inf \left\{\int_{\Om} |\nabla \varphi|^p \delta_{\Om}^{-\al} \dx  \biggm| \int_{\Om} |\varphi|^p \delta_{\Om}^{-(\al+p)} \dx\!=\!1, \varphi \in C_c^{\infty}(\Om) \right\} .
\end{align}
It is clear that $\mathbb{H}_{\al,p}(\Om) \geq 0$. As mentioned above, if $\al+p \leq 1$, then the $L^{\al,p}$-Hardy inequality does not hold for a bounded smooth domain (in other words,  $\mathbb{H}_{\al,p}(\Om)=0$). However, if  \eqref{Lp_Hardy} holds  for a given domain $\Gw$, then  it holds with $\mathbb{H}_{\al,p}(\Gw)$ as  the best constant for the inequality \eqref{Lp_Hardy}. We call the constant $\mathbb{H}_{\al,p}$  the {\em $L^{\al,p}$-Hardy constant} (or simply the weighted Hardy constant) of $\Om$. 

There is significant interest in finding or estimating the weighted Hardy constant $\mathbb{H}_{\al,p}(\Om)$ for various domains.  
We recall that if $\al+p>N$, then for {\em any} domain $\Om \subsetneq \R^N$, we have 
$\mathbb{H}_{\al,p}(\Om)\geq \left|\frac{\al+p-N}{p}\right|^{p}$ \cite[Theorem 5]{Avkhadiev} and \cite[Theorem 1.1]{GPP}. 
  Further, if $\Om \subsetneq \R^N$ is a convex domain and $\al+p>1$, then $\mathbb{H}_{\al,p}(\Om) = \left|\frac{\al+p-1}{p}\right|^{p}$ \cite[Theorem A]{Avkhadiev_sharp}. Many authors are interested to estimate this constant for certain non-convex domains, for instance, see \cite[lambda-close to convex domain]{Avkhadiev2022}, \cite[Annular domain in plane]{Avkhadiev2018}, \cite[domains with convex complement]{Avkhadiev_complement}, \cite[non-convex planar quadrilateral]{Tertikas}.
  
 Using the {\em supersolution construction}, Agmon trick, and Agmon-Allegretto-Piepen\-brink (AAP) type theorem (to be explained in sections~\ref{sec_pre} and \ref{sec_main}), we give the following lower bound of $\mathbb{H}_{\ga,p}(\Om)$ for a domain $\Om$ that admits a positive $p$-superharmonic function.
\begin{theorem}[{cf. \cite[Lemma~3.1]{LP} for the case $\ga=0$}]\label{Thm:easy}
Let $\Gw \subsetneq \R^N$ be a domain that admits a continuous positive $p$-superharmonic function $U$ (i.e., $-\De_p U\geq 0$) in $\Om$ such that $\frac{|\nabla U|}{U} \delta_{\Om} \geq \La_{\Om,U}>0$ a.e. in $\Om$. 
If $\La_{\Om,U}(p -1)>|\al| $, then 
$$\mathbb{H}_{\al,p}(\Om) \geq \mu_{\al,p,U}(\Om):= \left|\frac{\La_{\Om,U}(p -1)-|\al|}{p}\right|^p \,.$$ 
Moreover, the lower bound is sharp in the sense that for weakly mean convex domains (i.e., $-\De \delta_{\Om} \geq 0$) and $1-p<\alpha \leq 0$  we have $\mathbb{H}_{\al,p}(\Om) = \mu_{\al,p,U}(\Om)=|(\al+p-1)/p|^p$.
\end{theorem}  
  Using the above result, one can prove for $N-p<\al \leq 0$, the  well-known fact that for any domain $\mathbb{H}_{\al,p}(\Om)\geq \left(\frac{\ga+p-N}{p} \right)^p$ (Remark \ref{Rmk:Pinchover_Divya}). Moreover, as a consequence of this theorem, we establish the $L^{0,N}$-Hardy inequality in $\R^N \setminus S$, where $S=\{x\in \R^N: x_1>0, x_i=0 \ \forall i=2,\ldots,N\}$, see Corollary \ref{Cor:easy}. Note that this domain is not Lipschitz, and therefore,  this gives a simple  example of non-Lipschitz domain where the Hardy inequality holds.

The main aim of this article is to use the above theorem and give a lower estimate of the weighted Hardy constant for domains satisfying a {\em uniform exterior cone condition}. Let 
$$\mathcal{C}_{\beta}:=\{x \in \R^N: x_N \geq (\cos \beta)|x|\}$$ 
be the closed circular cone of aperture $\beta\in (0,\pi]$ with a vertex at the origin and axis in the direction of $e_N=(0,\ldots,1)$. Let $\Theta$ be a rotation in $\R^N$. We denote by $\mathcal{C}_{\beta}(z,\Theta) := z + \Theta (\mathcal{C}_{\beta})$ the rotated cone with a vertex at $z$ and aperture $\beta$.  
\begin{definition}{\em 
		We say that $\Gw \subset \R^N$ satisfies the  {\em $(\beta,r_0)$-uniform exterior cone condition}, if there exist $\beta \in (0,\pi]$ and $0<r_0\leq \infty$ such that for every $z \in \partial \Om$ there exists a rotation $\Theta$ such that
\begin{align} \label{UCC}
\overline{\Om} \cap B_{r_0}(z) \subset \mathcal{C}_{\beta}(z,\Theta) ,
\end{align}
}
\end{definition}
It is well
known that bounded Lipschitz domains satisfy the uniform exterior cone condition (see \cite{Kenig}). Notice that any convex domain satisfies the $(\frac{\pi}{2},\infty)$-uniform exterior cone condition. Moreover, any domain satisfying $(\beta,\infty)$-uniform exterior cone condition are simply connected.
For $N=2$, using Koebe one-quarter theorem, Ancona \cite{Ancona} showed that $\mathbb{H}_{0,2}(\Om) \geq 1/{16}$ if $\Gw$ is a simply connected domain. Laptev and Sobolev \cite{LS} established a more refined version of Koebe's theorem and obtained
a Hardy inequality which takes into account a quantitative measure of non-convexity. In particular,
they proved that if $\Om$ is a planar domain satisfying the $(\beta,\infty)$-uniform exterior cone condition, then $\mathbb{H}_{0,2}(\Om) \geq {\pi^2}/(16\beta^2)$. In \cite{Tertikas}, the authors provide a sharp lower bound for non-convex planar quadrilateral with exactly one non-convex angle. Note that the fact $N=2$, has been crucial in obtaining all these estimates, and it seems to be difficult to extend these results in higher dimensions.

In the following theorem, we provide a lower bound for the weighted Hardy constant $\mathbb{H}_{\al,p}(\Om)$ for a domain $\Gw \subsetneq\R^N$, $N\geq 2$,  satisfying the $(\beta,\infty)$-uniform exterior cone condition.  

\begin{theorem} \label{Thm:main}
Let $\Gw \subsetneq \R^N$ be a domain satisfying the $(\beta,\infty)$-uniform exterior cone condition, where $0<\gb<\pi$. Let $\gb<\gg< \pi$,  and consider the positive $p$-harmonic function $H$ in $\mathrm{int}(\mathcal{C}_{\gg})$ of the form $H(x)=H_\gg(r_x,\theta_x)=r_x^{\la} \Phi(\theta_x)$ ($r_x=|x|$, $\theta_x=\arccos(\braket{\frac{x}{|x|},e_N}$), where $\gl=\la(\gg)>0$ and $\Phi \in C^{\infty}([0,\gg])$ is strictly decreasing with $\Phi(0)=1$, $\Phi(\gg)=0$.  
Suppose that  $|\al|<\gl\Phi(\gb)^{1/\gl}(p-1)$, then  
$$\mathbb{H}_{\al,p}(\Om) \geq \mu_{\al,p,H}(\Om):=\left|\frac{\gl\Phi(\gb)^{1/\gl}(p-1)-|\al|}{p}\right|^p\,. $$
\end{theorem}
Theorems~\ref{Thm:easy} and \ref{Thm:main} give, in particular, a lower bound for the weighted Hardy constant $\mathbb{H}_{\al,p}(\mathcal{C}_{\beta})$ for circular cones with aperture $\beta \in (0,\pi)$ in any dimension. However, this lower bound is smaller than the known one in the case of $N=p=2$, and $\al=0$ (Remark \ref{Rmk:weak}-$(i)$). 

The organization of the paper is as follows. In section~\ref{sec_pre} we recall and prove some auxiliary results, and in Section~\ref{sec_main} we prove theorems \ref{Thm:easy} and \ref{Thm:main}, and some corollaries.
\medskip

Throughout the paper, we use the following notation and conventions:
\begin{itemize}

	\item We write $A_1 \Subset A_2$ if  $\overline{A_1}$ is a compact set, and $\overline{A_1}\subset A_2$.
	\item $C$ refers to  a positive constant which may vary from  line to line.
	\item We denote  $\De_{\al,p}(u):={\rm div} \left(\gd_\Gw^{-\ga} |\nabla u|^{p-2}\nabla u\right)$, called the $(\ga,p)$-Laplacian.
\end{itemize} 

\section{Preliminaries}\label{sec_pre}
First, we recall the Agmon-Allegretto-Piepenbrink (AAP)-type theorem (see \cite[Theorem~4.3]{Yehuda_Georgios} and the references therein) stated for our particular case.
\begin{theorem}[AAP-type theorem] \label{aap_thm} 
For any $\mu \in \R$,	the following assertions are equivalent:
	\begin{itemize}
		\item[(i)] The functional $\mathcal{Q}_{\ga,p,\gm}(\vgf)\!\!:=\!\!\displaystyle\int_{\Om}\!\! \left(|\nabla \varphi|^p \delta_{\Om}^{-\al} \! -\!  \mu |\varphi|^p  \delta_{\Om}^{-(\al+p)} \!\right)\!\dx$ is nonnegative on $C_c^\infty(\Omega)$.
		\item[(ii)] The equation $-\De_{\al,p} (u) - \frac{\mu}{\delta_{\Om}^{\al+p}}|u|^{p-2}u=0$ in $\Omega$ admits a positive (weak) solution.
		\item[(iii)] The equation $-\De_{\al,p} (u) - \frac{\mu}{\delta_{\Om}^{\al+p}}|u|^{p-2}u=0$ in $\Omega$ admits a positive (weak) supersolution.
	\end{itemize}		
\end{theorem}

Let $\Gw \subsetneq \R^N$ be a domain, and $\vge>0$.  Let 
$$\Gw_\vge:= \{x\in \Om \mid \delta_{\Om}(x)>\vge\}.$$
 For $\beta \in (0,\pi]$, there exists a positive $p$-harmonic function  in $\mathrm{int}(\mathcal{C}_{\gb})$ of the form $H(x)=H_{\beta}(r_x,\theta_x)=r_x^{\la} \Phi(\theta_x)$ ($r_x=|x|$, $\theta_x=\arccos(\braket{\frac{x}{|x|},e_N}$), where $\gl= \la(\beta)\geq 0$ and $\Phi \in C^{\infty}([0,\beta])$ is such that $\Phi(0)=1$, $\Phi(\gb)=0$, and $\Phi$ is strictly decreasing, see \cite{ALV,Veron}. In fact, there exists a unique positive $p$-harmonic function  in $\mathrm{int}(\mathcal{C}_{\beta})$ of the above form which vanishes on the boundary of the cone. Moreover, 
$\la(\pi)> 0$ if and only if $p+1>N$ \cite{ALV}. 
\begin{remark}  \label{Trans_harmonic} \rm
For $y \in \R^N$ and a rotation  $\Theta$, consider the closed cone $\mathcal{C}_{\gb}(y,\Theta)$. Using the translation invariance, it follows that there exists a positive $p$-harmonic function in $\mathrm{int}(\mathcal{C}_{\gb}(y,\Theta))$ of the form 
	\begin{equation}\label{eq_Hy}
	H_{y}(x)\!=\!|x-y|^{\la} \Phi(\theta_{y,\Theta, x})\,,
	\end{equation}
	where $\theta_{y,\Theta, x}=\arccos(\braket{\frac{x-y}{|x-y|},\Theta(e_N)}$ and $\la=\la(\beta)>0$ is the same as for the $p$-harmonic function $H=H_{\beta}$ in $\mathcal{C}_{\beta}$.
\end{remark}
Let $\Om$ be a domain in $\R^N$ satisfying the $(\beta,\infty)$-uniform exterior cone condition with $0<\gb<\pi$. Fix $\gb<\gg < \pi$, and let $H$ be the positive $p$-harmonic function  in $\mathrm{int}(\mathcal{C}_{\gg})$ of the form $H(x)=H_{\gamma}(r_x,\theta_x)=r_x^{\la} \Phi(\theta_x)$ ($r_x=|x|$, $\theta_x=\arccos(\braket{\frac{x}{|x|},e_N}$), where $\gl=\la(\gg)>0$ and $\Phi \in C^{\infty}([0,\gg])$ is such that $\Phi(0)=1$, $\Phi(\gamma)=0$, and $\Phi$ is strictly decreasing. Note that $\Phi(\gb)>0$. Then, for every $z \in \partial \Om$, \eqref{UCC} is satisfied. By Remark \ref{Trans_harmonic}, there exists a positive $p$-harmonic function $H_z(x)\!=\!|x-z|^{\la} \Phi(\theta_{z,\Theta,x})$ in $\Om\!\subset\! \mathrm{int}(\mathcal{C}_{\gg}(z,\Theta_z))$ satisfying \eqref{eq_Hy} with $y\!=\!z$. 
\begin{lemma} \label{Lem:UCP}
Let $\Gw \subsetneq \R^N$ be a domain satisfying the  $(\beta,\infty)$-uniform exterior cone condition with $\beta \in (0,\pi]$ if $p+1>N$ and $\beta \in (0,\pi)$ if $p+1\leq N$. For $y, z \in \partial \Om$, let $H_y, H_z$ be two positive $p$-harmonic functions  in $\Om$ of the form \eqref{eq_Hy}. If $H_y=H_z$ on a positive Lebesgue measurable set in $\Om$, then $H_y=H_z$ in $\Om$.
\end{lemma}
\begin{proof}
 Since $H_y,H_z$ are homogeneous function of degree $\gl>0$, it follows that  $\nabla H_y$ and $\nabla H_z $ do not  vanish in $\Om$. For instance, due to the homogeneity of $H_{y}$, we have \begin{align*} 
|\nabla H_{y}(x)|\geq \gl  |x-y|^{\gl-1} \Phi(\theta_{y,\Theta,x}) \ \ \mbox{in} \ 
\Om\,,
\end{align*}
and analogous estimate holds for $H_z$. Hence, for any $K \Subset \Om$, there exists $C_K>0$ such that $|\nabla H_y|, |\nabla H_z|\geq C_K>0$. Thus, the $p$-harmonic functions $H_y,H_z$ are real analytic in $\Om$ (see for example \cite[Theorem~1]{Lewis2}). By our assumption $H_y,H_z$ agree on a positive Lebesgue measurable set in $\Om$, therefore, $H_y=H_z$ in $\Gw$, see Appendix~\ref{app1} for a proof.
\end{proof}
\begin{proposition} \label{Prop:Divya_Cal}
Let $\Gw \subsetneq \R^N$ be a domain. For $x\in \Gw$, denote by $Px$ a point on $\partial \Gw$ such that  $\delta_{\Om}(x)=|x-Px|$. Then, for any $\varepsilon >0$ and $x \in \Om_{\varepsilon}$, there exists $\gt=\gt(\varepsilon,x)>0$ such that
\begin{align*}
\delta_{\Om}(y) \leq |y-Px| \leq (1+\varepsilon) \delta_{\Om}(y)   \qquad \forall y \in B_\gt(x) .  
\end{align*}
\end{proposition}
\begin{proof}
The first inequality holds trivially. For the second one, let $x \in \Om_{\varepsilon}$ and fix $Px \in \partial \Om$ such that  $\delta_{\Om}(x)=|x-Px|$. Note that $Px$ might not be unique.  Choose $\tau=t \delta_{\Om}(x)$  for some $0\!<\!t\!<\!1$ with $0\!<\!\frac{2t}{1-t} \!<\!\varepsilon$ such that $B_{\tau}(x) \!\subset\! \Om_{\varepsilon}$.
By the triangle inequality, for any $y \!\in\! B_{\tau}(x)$,
\begin{align} \label{eqn:1}
    |y-Px| \leq \delta_{\Om}(x) +\tau = (1+ t)\delta_{\Om}(x)   \,.
\end{align}
Also, notice that $\delta_{\Om}(x)-\tau\leq \delta_{\Om}(y)$, hence, 
$(1-t) \delta_{\Om}(x)\leq \delta_{\Om}(y)$.

Since $\frac{2t}{1-t}<\varepsilon$, it follow from \eqref{eqn:1} and the above inequality that 
$$\hspace{50mm} |y-Px| \leq \frac{1+t}{1-t} \delta_{\Om}(y) \leq (1+\varepsilon) \delta_{\Om}(y) \qquad \forall y \in B_{\tau}(x).   \hspace{10mm}\qedhere$$
\end{proof}

\section{Proof of the main results}\label{sec_main}
In this section, we prove theorems \ref{Thm:easy} and \ref{Thm:main}, and some of their corollaries. 
\begin{proof}[Proof of Theorem~\ref{Thm:easy}] 
Let $U\in C(\Gw)$ be a positive $p$-superharmonic function  in $\Om$ satisfying $\frac{|\nabla U|}{U} \delta_{\Om} \geq \La_{\Om,U}>0$ a.e. in $\Om$. Let $\{ K_n \}_{n \in \N}$ be a smooth exhaustion of $\Om$ such that $K_n \Subset K_{n+1}\Subset \Gw$.  By the AAP theorem (Theorem~\ref{aap_thm}), the existence of a positive supersolution of
\begin{align} \label{eq_pmu}
-\De_{\al,p} (u) - \frac{\mu}{\delta_{\Om}^{\al+p}}|u|^{p-2}u =0
\end{align} 
in a domain $D$ implies the existence of a positive solution of \eqref{eq_pmu} in $D$. Furthermore, using the Harnack convergence principle (\cite[Proposition~2.11]{Yehuda_Georgios}), it follows that the existence of positive solutions of \eqref{eq_pmu} in $K_{n}$ for any  $\mu<\mu_{\al,p,U}(\Om):=
 \left|\frac{-(|\al|/\La_{\Om,U})+p-1}{p}\right|^p \La_{\Om,U}^p$ and $n\in \N$, implies the existence of a positive solution of \eqref{eq_pmu}  in $\Gw$ with $\gm=\mu_{\al,p,U}(\Om)$, which by the AAP theorem (Theorem~\ref{aap_thm}), implies the weighted Hardy inequality in $\Gw$  with the constant $\mu_{\al,p,U}(\Om)$. 
Therefore, it is enough to show that \eqref{eq_pmu} 
admits a positive supersolution in $K_n$ for any $\mu < \mu_{\al,p,U}(\Om)$. 

 Since $U$ is continuous in $\overline{K_n}$, and $\La_{\Om,U}$ does not depend on the normalization of $U$, we may assume $U\ll1$ to be small enough in $K_n$.  Choose $\nu < \eta$, such that   $\nu , \eta \in (\frac{\hat{\al}+p-1}{p},\frac{\hat{\al}+p-1}{p-1}]$, where $\hat{\al}=-\frac{|\al|}{\La_{\Om,U} }$ (note that $\hat{\al}+p>1$).  For $a=\nu,\eta$, let $\la_{a}:=|a|^{p-2}a[\hat{\al}+(p-1)(1-a)]$. Note that $\gl_a \geq 0$ and $a\mapsto \la_{a}$ is strictly decreasing in $[\frac{\hat{\al}+p-1}{p},\frac{\hat{\al}+p-1}{p-1}]$.   

Now, using the supersolution construction and Agmon's trick, it follows that $U^{\nu}-U^{\eta}$
is a positive supersolution of the equation $-\De_{\al,p} (u) - \frac{\la_{\nu}\La_{\Om,U}^p }{\delta_{\Om}^{\al+p}}|u|^{p-2}u =0$  in  $K_n$.

 Indeed, by the chain rule \cite[Lemma~2.10]{DDP}, it follows that
\begin{align*}
&  -\De_{\al,p}(U^{\nu}-U^{\eta}) 
 \!=\! -|\nu U^{\nu -1}-\eta U^{\eta -1}|^{p-2} \bigg[(p-1)(\nu (\nu -1)U^{\nu -2}-\eta (\eta -1)U^{\eta -2} ) |\nabla U|^{2}  \\
 & \qquad \qquad \qquad \qquad \qquad \qquad \qquad \qquad - \al  (\nu U^{\nu -1}-\eta U^{\eta -1}) \bigg(\frac{\nabla U \cdot \nabla \delta_{\Om}}{\delta_{\Om}}\bigg) \bigg] \delta_{\Om}^{-\al} |\nabla U|^{p-2}  \\
& = U^{\nu (p-1)} |\nu - \eta U^{\eta -\nu}|^{p-2} \bigg[ (p-1)(\nu(1-\nu)-\eta(1-\eta)U^{\eta - \nu})  \\
 & \qquad \qquad \qquad \qquad \qquad \qquad \qquad \qquad + \al  (\nu -\eta U^{\eta -\nu}) \bigg(\frac{U}{|\nabla U|^2} \frac{\nabla U \cdot \nabla \delta_{\Om}}{\delta_{\Om}}\bigg) \bigg] \delta_{\Om}^{-\al} \frac{|\nabla U|^{p}}{U^p}.
 \end{align*}
 Hence, in $K_n$ we have
 \begin{align*}
&-\De_{\al,p}(U^{\nu}-U^{\eta}) \geq U^{\nu (p-1)} |\nu - \eta U^{\eta -\nu}|^{p-2} \bigg[ (p-1)(\nu(1-\nu)-\eta(1-\eta)U^{\eta - \nu})  \\
 &  \qquad \qquad \qquad \qquad - |\al|  (\nu -\eta U^{\eta -\nu}) \bigg(\frac{U}{|\nabla U|\delta_{\Om}} \bigg) \bigg] \delta_{\Om}^{-(\al+p)} \frac{|\nabla U|^{p}\delta_{\Om}^p}{U^p}\\
& \geq \! U^{\nu (p-1)} |\nu \!-\! \eta U^{\eta -\nu}|^{p-2} \bigg[\! (p\!-\!1)(\nu(1-\nu)-\eta(1-\eta)U^{\eta - \nu}) 
  - \frac{|\al|}{\La_{\Om,U}}  (\nu -\eta U^{\eta -\nu})  \bigg] \delta_{\Om}^{-(\al+p)} \La_{\Om,U}^p \\
 & = U^{\nu (p-1)} |\nu - \eta U^{\eta -\nu}|^{p-2} \left[ \frac{\la_{\nu}}{|\nu|^{p-2}} - \frac{\la_{\eta}}{|\eta|^{p-2}} U^{\eta - \nu}  \right]    \La_{\Om,U}^p\delta_{\Om}^{-(\al+p)} . 
\end{align*}
 Note that the first inequality uses the fact that $U\ll 1$ on $\overline{K_n}$.
Thus, in order to guarantee that $U^{\nu}-U^{\eta}$
is a positive suprsolution of the equation $-\De_{\al,p} (u) - \frac{\la_{\nu}\La_{\Om,U}^p}{\delta_{\Om}^{\al+p}}|u|^{p-2}u =0$  in $K_n$, it
suffices to impose the condition
\begin{align*}
U^{\nu (p-1)} |\nu - \eta U^{\eta -\nu}|^{p-2} \left[ \frac{\la_{\nu}}{|\nu|^{p-2}} - \frac{\la_{\eta}}{|\eta|^{p-2}} U^{\eta - \nu}  \right]    \geq \la_{\nu} (U^{\nu}-U^{\eta})^{p-1} 
\end{align*}
in $K_n$.  The latter inequality can be written as
\begin{align*} 
|\nu \!- \! \eta U^{\eta -\nu}|^{p-2} \! \! \left[ \frac{\la_{\nu}}{|\nu|^{p-2}} \!-\! \frac{\la_{\eta}}{|\eta|^{p-2}} U^{\eta - \nu}  \right]    \geq  \la_{\nu} (1-U^{\eta - \nu})^{p-1} \,.
\end{align*}
Expanding both sides we arrive at 
\begin{align*} 
[\la_{\nu}-AU^{\eta - \nu} +o(U^{\eta-\nu})  ]  \geq  \la_{\nu} (1 - (p-1)U^{\eta -\nu}+ o(U^{\eta -\nu})) \,,
\end{align*}
where $$A=(p-2)\frac{\la_{\nu} \eta}{\nu} + \la_{\eta} \frac{|\nu|^{p-2}}{|\eta|^{p-2}} \,. $$
 By a direct computation and the condition $0<\frac{\hat{\al}+p-1}{p} < \nu < \eta$,
one verifies that $A< \la_{\nu}(p-1)$, see Proposition~\ref{Prop:computation} in Appendix~\ref{app2}.
Therefore, $U^{\nu}-U^{\eta}$ is a positive supersolution of $-\De_{\al,p} (u) - \frac{\la_{\nu}\La_{\Om,U}^p }{\delta_{\Om}^{\al+p}}|u|^{p-2}u =0$ in $K_n$.  On the other hand,  $\la_{\nu}$ attains its maximum value $\left| \frac{\hat{\al}+p-1}{p}\right|^p$ at $\nu=\frac{\hat{\al}+p-1}{p}$. Thus, by taking $\nu \searrow \frac{\hat{\al}+p-1}{p}$, we obtain 
$$\mathbb{H}_{\al,p}(\Om) \geq \mu_{\al,p,U}(\Om)=
 \left|\frac{\La_{\Om,U}^p(p-1)-|\al|}{p}\right|^p .$$

By definition, if $\Gw$ is a mean convex domain, then $U=\delta_{\Om}$ is a positive $p$-superharmonic function, and  clearly $\La_{\Om,U}=1$. Hence, by the first part of Theorem~\ref{Thm:easy}, if $1-p<\al \leq 0$, then $\mathbb{H}_{\al,p}(\Om) \geq |(\al+p-1)/p|^p$. On the other hand, it is known (see \cite{Balinsky}) that if $\al+p-1>0$, then for mean convex domains $\mathbb{H}_{\al,p}(\Om) = |(\al+p-1)/p|^p$ . Therefore, the lower-bound is sharp. 
\end{proof} 

\begin{remark} \label{Rmk:Pinchover_Divya} \rm 
Let $\Om \subsetneq \R^N$ be {\em any} domain, and assume that $N-p<\ga \leq 0$.  Following the proof of \cite[Theorem~1.1 (taking $\al=0$)]{GPP}, we can construct a $p$-superharmonic function $U_{\varepsilon}$ in $\Om_{\varepsilon}$ satisfying
$$\frac{|\nabla U_{\varepsilon}| \delta_{\Om}}{U_{\varepsilon}} \geq \frac{1}{1+\varepsilon}\left[\frac{p-N}{p-1}\right] \ \ \mbox{a.e. in} \ \Om_{\varepsilon}$$
for every $\varepsilon>0$. Thus, using the above theorem, we obtain a modified proof \cite[Theorem 1.1]{GPP} that $\mathbb{H}_{\al,p}(\Om) \geq \left[\frac{\al+p-N}{p}\right]^p$ for $N-p<\ga \leq 0$.  Note that the equality occurs for $\ga=0$ and $\Gw$ is a punctured domain in $\R^N$.

\end{remark}
\begin{remark}{\em
Let $p=2$ and $\al=0$. For $0<\gb\leq 2\pi$ consider the sectorial region  
$$\mathcal{S}_\beta:= \{r \mathrm{e}^{\mathrm{i}\theta}  \mid 0<r<1, \  0<\theta< \gb\}\subset \R^2, \ \mbox{ where } {\mathrm{i}}^2=-1 .$$
In \cite[Theorem~4.1]{Davies} Davies  showed that there exists $\beta_{cr} \in (\pi, 2\pi)$ such that $\mathbb{H}_{0,2}(\mathcal{S}_\beta)={1}/{4}$ for all $\beta \leq \beta_{cr}$ and $\mathbb{H}_{0,2}(\mathcal{S}_\beta)$ decreases with $\beta$ for $\beta \geq \beta_{cr}$, and in the limiting case, $\mathbb{H}_{0,2}(\mathcal{S}_{2\pi})>0$, see also Remark \ref{Rmk:weak}. Note that $\mathcal{S}_{2\pi}$ satisfies the  $(\pi,\infty)$-uniform exterior cone condition. This shows again that, even being a non-Lipschitz  domain in $\R^2$, $\mathcal{S}_{2\pi}$ satisfies the $L^{0,2}$-Hardy inequality (as $\mathbb{H}_{0,2}(\mathcal{S}_{2\pi})>0$). As mentioned in the introduction, one can also conclude this, using either Ancona's result \cite{Ancona} or Laptev-Sobolev's result \cite{LS}. 
}
\end{remark}
As a simple consequence of Theorem~\ref{Thm:easy}, we obtain:
\begin{corollary} \label{Cor:easy}
Let $p\!=\!N$ and $\Gw= \mathcal{C}_{\pi}$ be the cone in $\R^N$ with aperture $\pi$. Then $\mathbb{H}_{0,p}(\mathcal{C}_{\pi})>0$.
\end{corollary}
\begin{proof}
From \cite[Theorem A]{ALV}, it follows that there exists a unique positive $p$-harmonic function  in $\mathcal{C}_{\pi}$ of the form $H(x)=H(r_x,\theta_x)=r_x^{\la_{\pi}} \Phi(\theta_x)$ ($r_x=|x|$, $\theta_x=\arccos(\braket{\frac{x}{|x|},e_N}$), where $\la_{\pi}=(p+1-N)/p=1/p$, and $\Phi \in C^{\infty}([0,\pi])$ is a decreasing function satisfying $\Phi(0)=1$, $\Phi(\pi)=0$. Take $U=H$ in Theorem \ref{Thm:easy}. Then, it follows from \cite[(3.21)]{ALV} that $\La_{\mathcal{C}_{\pi},U}>0$. Consequently, Theorem~\ref{Thm:easy} implies that $\mathbb{H}_{0,p}(\mathcal{C}_{\pi})>0$.
\end{proof}
\begin{remark} \rm
As a consequence of Corollary \ref{Cor:easy}, and  \cite[Theorem 3]{Lewis}, we infer that a half-line in $\R^N$ is uniformly $N$-fat.
\end{remark}

Observe that under the assumptions of Theorem~\ref{Thm:easy}, we obtain a lower bound for the best Hardy constant for the following Hardy inequality (with $\ga=0$):
\begin{align} \label{Eqn:var_Hardy}
\int_{\Om} |\nabla\varphi|^p \dx \geq C \int_{\Om} \frac{|\varphi|^p}{|x-x_{\partial \Om}|^{p}} \dx ,
\end{align}
where $x_{\partial \Om} \in \partial \Om$ is fixed, and $C>0$.
\begin{corollary}
Let $\Om$ be a domain in $\R^N$ satisfying at some point  $x_{\partial \Om} \in \partial \Om$ the $(\gb,\infty)$-exterior cone condition with $\beta \in (0,\pi]$ if $p+1>N$ and $\beta \in (0,\pi)$ if $p+1\leq N$. Then there exists $\la(\beta)>0$ such that the Hardy inequality \eqref{Eqn:var_Hardy} holds with $C=\left[ \frac{p-1}{p}\right]^p \la(\beta)^p$. Moreover, $\la(\beta) \geq \frac{p+1-N}{p}$ for any $\beta \in (0,\pi]$  when $p+1>N$.
\begin{proof}
Take $U=H_{x_{\partial \Om}}$ in Theorem \ref{Thm:easy}, where $H_{x_{\partial \Om}}$ is as in \eqref{eq_Hy}. Clearly, $U$ is a positive $p$-supermarmonic (in fact, harmonic) function in $\Om$. Further, due to the homogeneity of  $H_{x_{\partial \Om}}$, we have 
$$\frac{|\nabla U|}{U}\delta_{\Om} \geq \la(\beta) \,.$$
Thus, following the proof of Theorem \ref{Thm:easy}, we obtain \eqref{Eqn:var_Hardy} with $C\geq \left[ \frac{p-1}{p}\right]^p \la(\beta)^p$. Moreover, recall that for any $\beta \in (0,\pi]$, we have $\la(\beta) \geq \la_{\pi}$, and $\la_{\pi} = \frac{p+1-N}{p}$   when $p+1>N$.
\end{proof}

\end{corollary}

%

Next, using Theorem \ref{Thm:easy}, we proceed to prove our main theorem.

\begin{proof}[Proof of Theorem \ref{Thm:main}]
Since $\Gw$ satisfies the $(\gb,\infty)$-uniform exterior cone condition, and $0<\gb<\gg< \pi$, it follows that $\Phi(\gb)>0$.

Let $\{ K_n \}_{n \in \N}$ be a smooth exhaustion of $\Om$ such that $K_n \Subset K_{n+1}\Subset \Gw$.  Fix $n \in \N$ and choose $\varepsilon_0>0$ small enough such that $K_n \subset \Om_{\varepsilon_0}$ and also $-(|\al|/\la \Phi(\gb)^{1/\gl})(1+\varepsilon_0)+p>1$.  By Proposition \ref{Prop:Divya_Cal}, for any $0<\varepsilon<\varepsilon_0$, and all $x \in \Om_{\varepsilon}$, there exists $r=r(x,\varepsilon)>0$ such that
\begin{align} \label{Div_Cal}
\delta_{\Om}(y) \leq |y-Px| \leq (1+\varepsilon) \delta_{\Om}(y)  \qquad  \forall y \in B_r(x) \,.   
\end{align}
Now, since $\overline{K_n}$ is compact, there exist  $n_{\varepsilon} \in \N$ and $\{x_i\}_{i=1}^{n_{\varepsilon}}\subset \overline{K_n}$ such that
$K_n \subset \cup_{i=1}^{n_{\varepsilon}}B_{r_i}(x_i)$, where $r_i>0$ is such that \eqref{Div_Cal} holds on $B_{r_i}(x_i)$. For each $i\in\{1,\ldots,n_{\varepsilon}\}$, fix $Px_i \in \partial \Om$ such that $\delta_{\Om}(x_i)=|x-Px_i|$. Note that $K_n \subset \mathcal{C}_{\beta}(Px_i,\Theta_{Px_i})$ for all $i=1,\ldots,n_{\varepsilon}$, where $\Theta_{Px_i}$ are rotations in $\R^N$ depending on the point $Px_i$. Hence, each $H_{Px_i}$ is a positive supersolution of $-\De_p \varphi=0$ in $K_n$, where $H_{Px_i}$ is as in \eqref{eq_Hy}. 
 Using \cite[ Theorem 3.23]{HKM}, we infer that 
$$u_{\varepsilon}:=\min \{H_{Px_i} \mid  i=1,2,\ldots,n_{\varepsilon}\}$$ 
is a continuous positive supersolution of 
$-\De_p \varphi=0$ in $K_n$.    

We claim that 
\begin{align} \label{Eq:claim}
\frac{|\nabla u_{\varepsilon}(x)|}{u_{\varepsilon}(x)}\delta_{\Om}(x) \geq  \frac{\la \Phi(\gb)^{1/\la}}{1+\varepsilon} \qquad \mbox{for almost all } x\in K_n.
\end{align}
Indeed,  in view of Lemma \ref{Lem:UCP}, for a.e. $x \in K_n$ there exists a neighborhood $\mathcal{N}_x$ and a unique $H_{Px_j}$ ($j=j(x) \in \{1,2,...,n_{\varepsilon}\}$) such that $u_{\varepsilon}=H_{Px_j}$ in $\mathcal{N}_x$. Notice that  $H_{Px_j}$ is a homogeneous function of degree $\gl>0$. Hence,
\begin{align} \label{Homogenity}
|\nabla u_{\varepsilon}(x)|\geq \gl  |x-Px_{j}|^{\gl-1} \Phi(\theta_{Px_{j},\Theta_{Px_j},x})  \qquad \mbox{in } \ \
\mathcal{N}_x.
\end{align}
Since $x \in B_{r_i}(x_i)$ for some $i \in \{1,2,...,n_{\varepsilon}\}$, we get $|x-Px_i|^{\la}\leq (1+\varepsilon)^{\la}\delta_{\Om}(x)^{\la}$ from \eqref{Div_Cal}. On the other hand,  for any $j \neq i$, either 
$\Phi(\gb) |x-Px_j|^{\la}\leq (1+\varepsilon)^{\la}\delta_{\Om}(x)^{\la}$ or $\Phi(\gb) |x-Px_j|^{\la}> (1+\varepsilon)^{\la}\delta_{\Om}(x)^{\la}$. If the second possibility occurs, then 
$$H_{P_{x_i}}(x) \leq |x-Px_i|^{\la} \leq  (1+\varepsilon)^{\la}\delta_{\Om}(x)^{\la} <\Phi(\gb) |x-Px_j|^{\la} \leq H_{P_{x_j}}(x).$$ 
which contradicts the definition of $u_{\varepsilon}$. Thus,  we have 
\begin{align} \label{Est_1}
\Phi(\gb) |x-Px_j|^{\la}\leq (1+\varepsilon)^{\la}\delta_{\Om}(x)^{\la}
\end{align}
in $\mathcal{N}_x$. Now, using \eqref{Homogenity} and \eqref{Est_1}, we have 
\begin{align*}
\frac{|\nabla u_{\varepsilon}(x)|}{u_{\varepsilon}(x)}\delta_{\Om}(x) 
\geq \la\frac{\delta_{\Om}(x)}{|x-Px_j|} \geq \frac{\la \Phi(\gb)^{1/\la}}{1+\varepsilon} \quad \mbox{for almost all } x\in K_n.
\end{align*}
  This proves our claim \eqref{Eq:claim}.
 Note that estimate \eqref{Eq:claim} does not depend on the normalization of $u_\vge$ and we may assume that $u_\vge \ll 1$.  
 Now, using the supersolution construction and Agmon's trick as in Theorem \ref{Thm:easy} with $U=u_\vge$, $\Gl= \la \Phi(\gb)^{1/\la}/(1+\varepsilon)$, and 
 $\al_{\varepsilon}=\frac{(1+\varepsilon)\al}{\la \Phi(\gb)^{1/\la}}$, we obtain that 
the best Hardy-constant for the operator $-\Gd_{\ga,p}$ with respect to the Hardy-weight $\gd_\Gw^{-(\ga+p)}$ in $K_n$ is at least
$$ \left|\frac{p-1-|\al_{\varepsilon}|}{p}\right|^p \frac{\la^p \Phi(\gb)^{p/\la}}{(1+\vge)^p} \, .$$ 
 Thus, by taking $\varepsilon \ra 0$, it follows that
 $$\hspace{30mm} \mathbb{H}_{\al,p}(\Om) \geq \mu_{\al,p,H}(\Om)=
 \left|\frac{\la \Phi(\gb)^{1/\gl}(p-1) -|\al|}{p}\right|^p . \hspace{30mm}\qedhere$$
 \end{proof}
 
\begin{remark} \label{Rmk:weak} \rm
$(i)$ Let $p+1>N$ and $\Om$ be a domain in $\R^N$ satisfying the $(\beta,\infty)$-uniform exterior cone condition with $\gb \in (0,\pi)$. Consider the unique positive $p$-harmonic function $H$ in $\mathcal{C}_{\pi}$ of the form $H(x)=H_{\pi}(r_x,\theta_x)=r_x^{\la} \Phi(\theta_x)$ ($r_x=|x|$, $\theta_x=\arccos(\braket{\frac{x}{|x|},e_N}$), where $\la=\la(\pi)= (p+1-N)/{p}>0$ and $\Phi \in C^{\infty}([0,\pi])$ is strictly decreasing in $\theta$ and $\Phi(0)=1$, $\Phi(\pi)=0$, see \cite{ALV} for the existence of such a $p$-Harmonic function. If 
$\la \Phi(\beta)^{1/\la}(p-1)>|\al|$,
then  
$$\mathbb{H}_{\al,p}(\Om) \geq \mu_{\al,p,H}(\Om)=\left|\frac{\la \Phi^{1/\gl}(\beta)(p-1) -|\al|}{p}\right|^p .$$ 
In particular, for $\al =0$ we obtain 
 $\mathbb{H}_{0,p}(\Om) \geq \mu_{0,p,H}(\Om)=\left|\frac{p-1}{p}\right|^p \la(\pi)^p \Phi(\beta)^{p/\la}$. Furthermore, for $p=2$ and $N=2$, one can show that $\Phi(\theta)=\cos(\theta/2)$. Hence, $\mathbb{H}_{0,2}(\Om) \geq \cos(\beta/2)^{4}/16$.
 
$(ii)$ In light of the above remark, we have  for a domain $\Om \subsetneq\R^N$ satisfying the $(\beta,\infty)$-uniform exterior cone condition ($\beta \in (0,\pi)$) that
 $\mathbb{H}_{0,p}(\Om) \geq \mu_{0,p,H}(\Om)=\left|\frac{p-1}{p}\right|^p \la(\pi)^p \Phi(\beta)^{p/\la}$ when $p+1>N$. Recall that for any arbitrary domain $\Om$ and $p>N$, it is known that  $\mathbb{H}_{0,p}(\Om) \geq \left|\frac{p-N}{p}\right|^p $,  see for example \cite{GPP}. On the other hand, for $p>N$, $\gl=\la(\pi)= (p+1-N)/p$, and hence, $\mu_{0,p,H}(\Om)=\left|\frac{p-1}{p}\right|^p \left[\frac{p+1-N}{p}\right]^p \Phi(\beta)^{p/\la}$. Notice that when $p \searrow N$, the lower bound $\mu_{0,p,H}(\Om)$ is better than $\left|\frac{p-N}{p}\right|^p$.
\end{remark}

\begin{center}
	{\bf Acknowledgments}
\end{center}
The authors acknowledge the support of the Israel Science Foundation \!(grant $637/19$) founded by the
Israel Academy of Sciences and Humanities. U.D. is also supported in
part by a fellowship from the Lady Davis Foundation.
\appendix
\section{} \label{app1}
In this appendix we show that if $u$ and $v$ are two real analytic functions in an open box $B=\Pi_{i=1}^N(a_i,b_i)\subset \R^N$, and  $u=v$ on a Lebesgue measurable set $A\subset B$ with a positive measure, then  $u=v$ in $\Gw$.

Indeed, let $f=u-v$, then $f$ is a real analytic function in $B$, vanishing   on $A$. We need to show that $f=0$ in $B$. We prove this claim by induction on the dimension $N$. 

If $N=1$, then $B$ is an open interval $(a,b)$ and therefore, there exists $n_0$ such that the set $A\cap [a+1/n_0,b-1/n_0]$ is infinite, hence this set has a limit point, and by the classical identity theorem $f=0$ in $B$.

Suppose that that the claim is true for boxes in $\R^N$, where $N\geq 1$, and consider an open box $B_{N+1}=B_N\times B_1\subset \R^{N+1}$. We write points in   
$B$ as $(x,t)$
with $x\in B_N$ and $t\in B_1$.

For each fixed $t\in B_1$
the function $x\mapsto f(x,t)$ is real analytic on $B_N$,
and for each fixed $x\in B_N$
the function $t\mapsto f(x,t)$ is real analytic on $B_1$.

Consider a real analytic $f$ on the box $B_{N+1}$ such that 
$f=0$ on a measurable set $A\subset B_{N+1}$ such that $|A|_{N+1}>0$.

By Fubini's theorem,
$$|A|_{N+1}=\int_{B_1}\int_{B_N} \chi_A(x,t)\,\dx\,\dt.$$
Since $|A|_{N+1}>0$, the inner integral above is positive  for $t\in A_1\subset B_1$
with $|A_1|_1>0$. By the induction hypothesis, for each $t\in A_1$
the function $x\mapsto  f(x,t)$ vanishes identically on $B_N$.
Hence, $f=0$ on $B_N\times A_1$. So, for each fixed $x\in B_N$, the mapping $t\mapsto f(x,t)$ vanishes
on $A_1\subset B_1$, and recall that $|A_1|_1>0$. By our proof for $N=1$, 
this function vanishes on $B_1$.
It follows that $f=0$ on $B_N\times B_1=B_{N+1}$.

\section{Auxiliary inequalities}\label{app2}
In the section we prove an elementary inequality that plays crucial role in the proofs of theorems~~\ref{Thm:easy} and \ref{Thm:main}.
\begin{proposition} \label{Prop:computation} Let $a \in \R$ and $\nu < \eta$ be such that,  either $\nu,\eta \in [\frac{a+p-1}{p},\frac{a+p-1}{p-1}]$ if $a+p>1$, or $\nu,\eta \in [\frac{a+p-1}{p},0]$ if $a+p<1$ respectively. Then 
	\begin{align} \label{to_prove1}
	(p-2)\frac{\la_{\nu} \eta}{\nu} + \la_{\eta} \frac{|\nu|^{p-2}}{|\eta|^{p-2}} <  \la_{\nu} (p-1) \,,
	\end{align}
	where $\la_{t}=|t|^{p-2}t[a+(1-t)(p-1)]$ with $t=\nu,\eta$.
\end{proposition}
\begin{proof} Since $\la_{\gn}>0$, it follows that  \eqref{to_prove1} holds if 
	\begin{align} \label{s1}
	(p-2)\frac{\eta}{\nu} + \frac{\la_{\eta}}{\la_{\nu}} \frac{|\nu|^{p-2}}{|\eta|^{p-2}} <   (p-1) \,.
	\end{align}
	An elementary computation shows that \eqref{s1} is true if
	$$\left(\frac{\eta}{\nu}\right) \frac{-1[a+(1-\nu)(p-1)] +[a+(1-\eta)(p-1)]}{a+(1-\nu)(p-1)} < \frac{(p-1)(\nu-\eta)}{\nu} \,, $$ 
	which is equivalent to 
	\begin{align} \label{s2}
	\left(\frac{\eta}{\nu} \right) \frac{(\nu-\eta)}{a+(1-\nu)(p-1)} < \frac{(\nu-\eta)}{\nu} \,.
	\end{align}
	Since $\nu<\eta$ and $\nu [a+(1-\nu)(p-1)]$ is nonnegative, \eqref{s2} holds if
	\begin{align} \label{s3}
	\eta > a+(1-\nu)(p-1) \,,
	\end{align}
	which is indeed true as  $a+(1-\nu)(p-1)\leq \nu$.
\end{proof}


\end{document}